\documentclass[10pt]{article}
\usepackage{latexsym}
\usepackage{amssymb}
\usepackage{hhline}
\usepackage{graphicx,amssymb,lineno, amsmath}
\usepackage{multicol}
\usepackage{float}
\usepackage{verbatim}
\usepackage[all]{xy}

by 2 \advance \topmargin -1.8in \leftmargin 210mm \advance
\leftmargin -\textwidth \divide \leftmargin by 2 \advance
\leftmargin

\newcommand{\ra}{\rightarrow}

\newtheorem{theorem}{Theorem}
\newtheorem{lemma}{Lemma}

\newenvironment{proof}{\noindent{\bf Proof:}}{\hspace*{1mm} \hfill$\Box$ \par \bigskip }

\date{}
\title{\bf On Some\\
Multicolor Ramsey Numbers\\
Involving $K_3+e$ and $K_4-e$}

\begin{document}

\thispagestyle{empty}
\maketitle \footnote[0]{Research supported by the NSF Research Experiences for Undergraduates Program (\#1062128) held at the Rochester Institute of Technology during the summer of 2011.}
\thispagestyle{empty}

\vspace{-18mm}
\begin{center}

\begin{multicols}{2}
{\bf Daniel S. Shetler}\\
 Department of Mathematics\\
 Whitworth University\\
Spokane, WA 99251\\
 {\tt dshetler12@my.whitworth.edu}\\[4mm]

{\bf Michael A. Wurtz}\\
 Department of Mathematics\\
 Northwestern University \\
 Evanston, IL 60208 \\
 {\tt wurtz@u.northwestern.edu}\\[4mm]
\end{multicols}

{\bf Stanis\l{}aw P. Radziszowski}\\
 Department of Computer Science \\
 Rochester Institute of Technology \\
 Rochester, NY 14623 \\
 {\tt spr@cs.rit.edu}\\[4mm]

\end{center}

\vspace{3mm}


\noindent{\bf Abstract:} The Ramsey number $R(G_1, G_2, G_3)$ is the smallest
positive integer $n$ such that for all 3-colorings of the edges of $K_n$
there is a monochromatic $G_1$ in the first color, $G_2$ in the second color,
or $G_3$ in the third color. We study the bounds on various 3-color
Ramsey numbers $R(G_1, G_2, G_3)$, where
$G_i \in \{K_3, K_3+e, K_4-e, K_4\}$. The minimal
and maximal combinations of $G_i$'s correspond to
the classical Ramsey numbers $R_3(K_3)$ and $R_3(K_4)$,
respectively, where $R_3(G) = R(G, G, G)$. Here, we focus
on the much less studied combinations between these two cases.

Through computational and theoretical means we establish
that $R(K_3, K_3, K_4-e)=17$, and by construction we raise
the lower bounds on $R(K_3, K_4-e, K_4-e)$ and
$R(K_4, K_4-e, K_4-e)$. For some $G$ and $H$ it was
known that $R(K_3, G, H)=R(K_3+e, G, H)$; we prove this
is true for several more cases including
$R(K_3, K_3, K_4-e) = R(K_3+e, K_3+e, K_4-e)$.

Ramsey numbers generalize to more colors, such as in the famous
4-color case of $R_4(K_3)$, where monochromatic triangles are avoided.
It is known that $51 \leq R_4(K_3) \leq 62$.
We prove a surprising theorem stating that if
$R_4(K_3)=51$ then $R_4(K_3+e)=52$, otherwise $R_4(K_3+e)=R_4(K_3)$.

\section{Introduction}

For undirected simple graphs
$G_1, \dots, G_m$, a $(G_1, \dots, G_m)$-coloring is a partition
of the edges of a complete graph into $m$ colors such that no
color $i$ contains $G_i$ as a subgraph. A $(G_1, \dots, G_m;n)$-coloring
is a $(G_1, \dots, G_m)$-coloring of $K_n$. Further,
$\mathcal{R}(G_1, \dots, G_m)$ and $\mathcal{R}(G_1, \dots, G_m;n)$
will denote the sets of all corresponding colorings.
The Ramsey number $R(G_1, \dots, G_m)$ is defined as
the minimum number of vertices $n$ such that no
$(G_1, \dots, G_m;n)$-coloring exists. Note that a
standard graph can be considered as a 2-coloring of
the edges of a complete graph, where the edges of
the graph are those in the first color. As such,
we will call a $(G_1,G_2)$-coloring a $(G_1,G_2)$-good
graph. The known values and bounds for various types
of Ramsey numbers are compiled in the dynamic survey
{\em Small Ramsey Numbers} by the third author \cite{Rad}.\\

\noindent We will use the following notation throughout the paper:
\vspace{-2mm}
\begin{align*}
N_c(v) &: \text{neighborhood of vertex $v$ in color $c$} \\
G-v &: \text{coloring or graph induced by } V(G) \setminus \{v\} \\
G \setminus \{u, v\} &: G \text{ without edge } \{u, v\} \\
J_n &: K_n-e, \text{ equal to } K_n \text{ with one edge deleted} \\
K_n+e &: K_n \text{ connected to an additional vertex by one edge} \\
R_n(G) &: n\text{-color Ramsey number } R(G, \dots, G) \\
\text{colors} &: \text{we refer to consecutive colors corresponding to the parameters}
\end{align*}

\vspace{-4mm}
\indent\indent\indent \ \  of Ramsey colorings as red, green, blue, and yellow\\

We will be using the Ramsey arrowing operator $\ra$. We say that $F \ra (G_1, \dots, G_m)$ holds iff for all partitions of the edges of $F$ into $m$ colors $F_1, \dots, F_m$ there exists $G_i \subseteq F_i$ for some $1 \leq i \leq m$. The Ramsey number $R(G_1, \dots, G_m)$ can also be defined using the arrowing operator as the smallest $n$ such that $K_n \ra (G_1, \dots, G_m)$.\\

Observe that if $H'$ is a subgraph of $H$, then any $(G, H')$-good graph is also a $(G, H)$-good graph. Thus $\mathcal{R}(G,H') \subseteq \mathcal{R}(G, H)$, and therefore $R(G, H') \leq R(G, H)$. The complement of a $(G,H)$-good graph is an $(H,G)$-good graph, hence $R(G, H)=R(H, G)$. This monotonicity and symmetry of 2-color Ramsey numbers extend to multiple colors.\\

In what follows we discuss Ramsey numbers for parameters between $(K_3, K_3, K_3)$ and $(K_4,K_4,K_4).$ In this range there are four classical Ramsey numbers $R(K_p, K_q, K_r)$ of which only one exact value $R(K_3, K_3, K_3)=17$ is known \cite{Rad}.
Arste, Klamroth, and Mengersen \cite{AKM} studied a variety
of 3-color Ramsey numbers
$R(G_1, G_2, G_3)$ for $G_i$'s on at most four vertices. Several of
the cases still unsolved fall within the $(K_3, K_3, K_3)$ to $(K_4, K_4, K_4)$
range. Figure \ref{fig:RFlow} below is presented as a poset of possible
parameters ordered coordinate-wise under inclusion for
$G_i\in  \{K_3, J_4, K_4\}$. The only two numbers known
in this range are $R(K_3, K_3, K_3)=R(K_3, K_3, J_4)=17$
(Theorem \ref{thm:33J}). For the open cases the best known
bounds are presented. The Ramsey numbers with at least one
parameter involving $K_3+e$ are studied in Section \ref{sec:K3}.

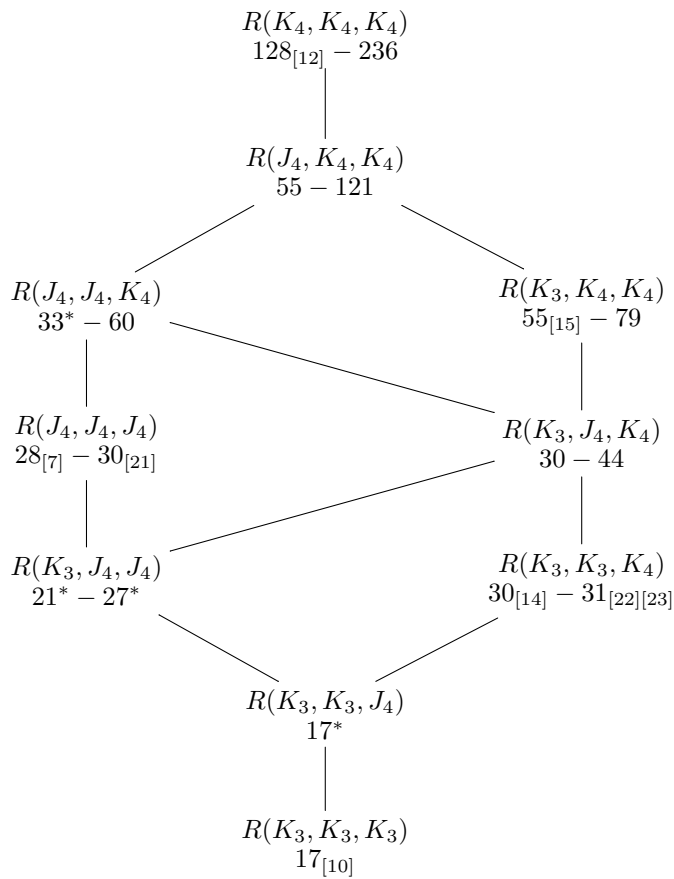
\begin{figure}[H]
\begin{center}

\begin{displaymath}
 \xymatrix{
 \nonumber & *\txt{$R(K_4,K_4,K_4)$\\ $128_{[12]} - 236$} \ar@{-}[d]\\
 \nonumber & \txt{$R(J_4,K_4,K_4)$\\ $55 - 121$} \ar@{-}[dl] \ar@{-}[dr]\\
 \txt{$R(J_4,J_4,K_4)$\\ $33^* - 60$} \ar@{-}[d] \ar@{-}[drr]& \nonumber  & \txt{$R(K_3,K_4,K_4)$\\ $55_{[15]} - 79$} \ar@{-}[d] \\
       \txt{$R(J_4,J_4,J_4)$\\ $28_{[7]} - 30_{[21]}$} \ar@{-}[d] & \nonumber  &  \txt{$R(K_3,J_4,K_4)$\\ $30 - 44$} \ar@{-}[dll] \ar@{-}[d]\\
        \txt{$R(K_3,J_4,J_4)$\\ $21^* - 27^*$} \ar@{-}[dr] & \nonumber &   \txt{$R(K_3,K_3,K_4)$\\ $30_{[14]} - 31_{[22][23]}$} \ar@{-}[dl] \\
       \nonumber &  \txt{$R(K_3,K_3,J_4)$\\ $17^*$}\ar@{-}[d]\\
       \nonumber &  \txt{$R(K_3,K_3,K_3)$\\ $17_{[10]}$}\\
       }
\end{displaymath}

\caption{Ramsey numbers for parameters between $(K_3, K_3, K_3)$ and $(K_4, K_4, K_4)$.
The results of this paper are marked with a *, and the bounds without references are
obtained by monotonicity or by application of
the standard upper bound (see 6.1.a in \cite{Rad})
to the bounds for smaller parameters in this figure or
to the results listed in \cite{AKM}.}
\label{fig:RFlow}
\end{center}
\end{figure}

\bigskip
\section{From $K_3$ to $K_3+e$}\label{sec:K3}
In the case of two colors, Burr, Erd\H{o}s,
Faudree, and Schelp \cite{BEFS} proved that,
for $m,n \geq 3$ and $m+n \geq 8$, $R({\widehat K}_{m,p},
\widehat{K}_{n,q})=R(K_m, K_n)$, where $p=\lceil m/(n-1) \rceil$,
$q=\lceil n/(m-1) \rceil$,
and ${\widehat K}_{k,l}=K_{k+1}-K_{1,k-l}$, the graph obtained
from a $K_k$ by adding a vertex adjacent to $l$ vertices in $K_k$.\\

\bigskip
For more colors, it has been proven that in some cases adding an edge to $K_3$ leaves Ramsey numbers unchanged, such as the following:

\bigskip
\begin{itemize}
\item $R_3(K_3+e)=R_3(K_3)=17$ \cite{YR3},
\item $R(K_3+e,K_3+e,K_4)=R(K_3,K_3,K_4)$ \cite{AKM}.
\end{itemize}

Several similar cases are presented in \cite{AKM}. We give further evidence of such behavior by establishing three new cases. This raises the question of when the parameter $K_3$ can be extended to $K_3+e$ without changing the Ramsey number.\\

In Theorem \ref{thm:33J} of the next section we will prove that $R(K_3,K_3,J_4)=17$.
This result will be used in the proof of the following Theorem \ref{thm:33Je}.

\bigskip
\begin{theorem}\label{thm:33Je}
$R(K_3, K_3, J_4) = R(K_3+e, K_3+e, J_4)$ $[=17]$.
\end{theorem}
\begin{proof}
By Theorem \ref{thm:33J} and monotonicity
of Ramsey numbers we have that
$17=R(K_3, K_3, J_4) \leq R(K_3+e, K_3+e, J_4)$.
Assume towards a contradiction that
$R(K_3, K_3, J_4) < R(K_3+e, K_3+e, J_4)$, and let
$G$ be a $(K_3+e, K_3+e, J_4; 17)$-coloring.
We may assume without loss of generality that there is a red
$K_3$ in $G$ with vertices $\{v_1, v_2, v_3\}$. Let the graph $H$ be the red component of $G$ induced by $V(G) \setminus \{v_1, v_2, v_3\}$. Clearly $H$ contains no $K_3+e$. Also, $H$ cannot contain a $\overline{K_5}$, since otherwise
together with $\{v_1, v_2\}$ it would span
a green and blue ${J_7}$ in $G$. By Lemma \ref{lem:J7} of the next
section, $J_7 \ra (K_3+e, J_4)$, which is a contradiction. So $H$
is a $(K_3+e, K_5; 14)$-good graph, which is impossible since
$R(K_3+e, K_5)=14$ \cite{Clan}.
\end{proof}

In the known non-trivial cases it appears that extending the parameter
$K_3$ to $K_3+e$ does not change Ramsey numbers. Irving \cite{Ir} stated
that for $k>2$, it seems likely that $R_k(K_3+e) = R_k(K_3)$. The following
theorem may add credence to or disprove this statement. It is known that
$51 \leq R_4(K_3) \leq 62$ \cite{Chu1}\cite{FKR}.

\bigskip
\begin{theorem} \label{thm:R4(3)}
\ \\
$(a)$ If $R_4(K_3)=51$, then $R_4(K_3+e)=R(K_3, K_3, K_3, K_3+e)=52$, and \\
$(b)$ If $R_4(K_3)>51$, then $R_4(K_3+e)=R_4(K_3)$.
\end{theorem}
\begin{proof}
Suppose $H$ is a $(K_3+e, K_3+e, K_3+e, K_3+e;n)$-coloring for some $n \geq R_4(K_3)$ and $n\geq52$. Then we may assume without loss of generality that $H$ contains a red $K_3$. Let $v$ be a vertex of this $K_3$, then we may also assume that $|N_g(v)| \geq \lceil (n-3)/3 \rceil\geq17$.
Note that the green color cannot occur in $N_g(v)$.
However, $N_g(v)$ induces a $(K_3+e, K_3+e, K_3+e)$-coloring,
and since $R_3(K_3+e)=17$ \cite{YR3}, $|N_g(v)| \leq 16$. This gives rise to a contradiction, and thus proves $(b)$ and the upper bound for $(a)$. What remains to be shown is the lower bound in $(a)$.

We construct a $(K_3, K_3, K_3, K_3+e; 51)$-coloring $C_{51}$ by extending the well known Chung $(K_3, K_3, K_3, K_3; 50)$-coloring $C_{50}$ \cite{Chu1}. Partition the set of vertices of $C_{50}$ as $V=R \cup G \cup B \cup \{x,y\}$, where $|R|=|G|=|B|=16$. The edge $\{x,y\}$ is yellow, edges in $\{\{x,v\},\{y,v\} : v \in R\}$ are red, edges in $\{\{x,v\}, \{y,v\} : v \in G\}$ are green, and edges in $\{\{x,v\},\{y,v\}:v\in B\}$ are blue. Each of $R,G,$ and $B$ induces a
$(K_3,K_3,K_3;16)$-coloring, where the first has no red edges, the second
no green edges, and the third no blue edges. Chung also described a way to color the edges between $R,G$, and $B$ without forming a monochromatic $K_3$. We omit the details as they are irrelevant to our proof. The additional vertex $z$ is connected to $R,G$, and $B$ in the same way as $x$ and $y$, and the edges $\{x,z\}$ and $\{y,z\}$ are yellow. This $C_{51}$ on the vertex set $V \cup \{z\}$ has exactly one monochromatic $K_3$, namely an isolated yellow $K_3$ on $\{x,y,z\}$.
Thus, easily,
$C_{51}$ is a $(K_3, K_3, K_3, K_3+e; 51)$-coloring.
By the monotonicity of Ramsey numbers, $(a)$ follows.
\end{proof}

We close this section with a case where a similar but
unconditional equality can be proven even when the Ramsey
number is unknown, namely for the case
$30\leq R(K_3,K_3,K_4)\leq31$ \cite{Ka2}\cite{PR1}\cite{PR2}.
Our next theorem improves on the old result that
$R(K_3, K_3, K_4) = R(K_3+e, K_3+e, K_4)$ \cite{YR3}\cite{AKM}.
In the following, $P_k$ will denote a path on $k$ vertices.

\medskip
\begin{theorem} \label{thm:R(334)}
$R(K_3, K_3, K_4) = R(K_3+e, K_3+e, K_5-P_3)$.
\end{theorem}
\begin{proof}
Let $n=R(K_3, K_3, K_4)$, and assume towards a contradiction that $G$ is a $(K_3+e, K_3+e, K_5-P_3; n)$-coloring. By the remarks above we know that $30 \leq n \leq 31$. There is a blue $K_4$ in $G$, let its vertices be $\{v_1, v_2, v_3, v_4\}$. If $c$ is red or green and $|N_c(v_i)|>2$, then $N_c(v_i)$ induces a $(K_3+e, K_5-P_3)$-good graph.  Since $R(K_3+e, K_5-P_3)=10$ \cite{Clan}, then
in both cases
$N_c(v_i)$ has order at most $9$. Let $N_i=N_b(v_i) \setminus \{v_1, v_2, v_3, v_4\}$, then $|N_i| \geq (n-4)-2\cdot9 \geq 8$. With four such $N_i$'s covering $n-4$ vertices, some vertex $v$ must be contained in at least $2$ of them. Then $\{v, v_1, v_2, v_3, v_4\}$ forms a blue $K_5-P_3$ in $G$.
\end{proof}

\section{Ramsey Number $R(K_3,K_3,J_4)$} \label{sec:33J4}

The smallest open case for complete graphs in Figure \ref{fig:RFlow} is $R(K_3,K_3,K_4)$, of which the current bounds of $30$ and $31$ have not been improved since 1998 \cite{PR1}. Obtaining the exact value has continued to remain beyond the reach of computational methods. In this section we prove that the Ramsey number $R(K_3, K_3, K_4 - e)$ is equal to 17, considering it as an intermediate step between $R(K_3,K_3,K_3)$ and the solution to the elusive $R(K_3,K_3,K_4)$. The proof of $R(K_3,K_3,J_4)=17$
needs some lemmas, which are then used in two different computational approaches.

\begin{lemma} \label{lem:hex}
Every $(K_3+e,J_4;6)$-good graph contains a $C_6$ or $G=2K_3$.
\end{lemma}

\begin{proof}
Suppose that $G$ is a $C_6$-free $(K_3+e,J_4;6)$-good graph with vertices $\{v_1,v_2,v_3,v_4,v_5,v_6\}$ that is not equal to $2K_3$. First, we show that $G$ cannot contain a $K_3$: assume $\{v_1, v_2, v_3\}$ forms a $K_3$. Then for $i \in \{1, 2, 3\}$ and $j \in \{4, 5, 6\}$, $\{v_i, v_j\}$ cannot be an edge. This means that $\{v_4, v_5, v_6\}$ must induce a $K_3$ to avoid a $\overline{J_4}$, and the resulting graph is equal to $2K_3$.

We can now assume that $G$ is a $(K_3,J_4)$-good graph.  Next, we consider the cases with respect to the length of the longest path in $G$ as follows.

\begin{description}
 \item[($P_6$)] Assume that the longest path is $P_6=v_1v_2v_3v_4v_5v_6$.
To prevent $\{v_1,v_3,v_4,v_6\}$ from forming a $\overline{J_4 }$,
the graph $G$ must contain either edge $\{v_1,v_4\}$ or $\{v_3,v_6\}$, since any other additional edge would form a $C_6$ or $K_3$. Thus, without loss of generality, we can assume that $\{v_3,v_6\}\in E(G)$. Similarly, $G$ must contain an additional edge in $\{v_1,v_2,v_4,v_6\}$, otherwise there is a $\overline{J_4}$. Any such edge forms a $K_3$ or a $C_6$, which is a contradiction.
 
 \item[($P_k)$] Assume that the longest path $P$ in $G$ is of length $k<6$. By considering one or two forbidden $\overline{J_4}$'s, it can be shown that $P$ together with additional edges would contain $P_{k+1}$, leading to a contradiction. We leave the details for the reader to verify.
\end{description}

\end{proof}

\begin{lemma} \label{lem:J7}
$J_7 \ra (K_3+e,J_4)$.
\end{lemma}

\begin{proof}
Suppose that there is a coloring $C$ of $J_7=K_7 \setminus \{x,y\}$ witnessing the contrary. Let $G$ be the graph formed by the edges of the first color of $C$. Then $G$ contains no $K_3+e$ and $\overline{G} \setminus \{x,y\}$ contains no $J_4$. Further, $G - x$ is a $(K_3+e,J_4;6)$-good graph, and by Lemma \ref{lem:hex} it contains a $C_6$ or is equal to $2K_3$.

First assume that
$G - x$ contains a $C_6 = v_1v_2v_3v_4v_5y$ as shown in Figure \ref{fig:hexProof}.
Note that $\{v_1,v_3,v_5\}$ and $\{v_2,v_4,y\}$ are independent sets.
To avoid $J_4$ in $\overline{G} \setminus \{x,y\}$
on vertices $\{v_2,v_4,x,y\}$,
the graph $G$ must contain at least one of $\{x,v_2\}$ or $\{x,v_4\}$.
Without loss of generality assume that $\{x,v_2\}$ is in the graph.
Now to avoid $\overline{J_4}$ on the set $S=\{v_1,v_3,v_5,x\}$, $G$ must contain
at least two edges with both endpoints in $S$. However, any two such edges would complete a $K_3+e$, a contradiction. On the other hand, suppose $G - x = 2K_3$.
Any edge from $x$ to $2K_3$ would form a $K_3+e$,
so all $6$ edges must be in $\overline {G}$, but this leads to a
$J_4$ in $\overline{G} \setminus \{x,y\}$.
\end{proof}

The above shows that $J_7 \ra (K_3+e,J_4)$. Note that this also easily implies $R(K_3+e,J_4)=7$.

\begin{figure}[h]
\begin{center}
\includegraphics[scale=.28]{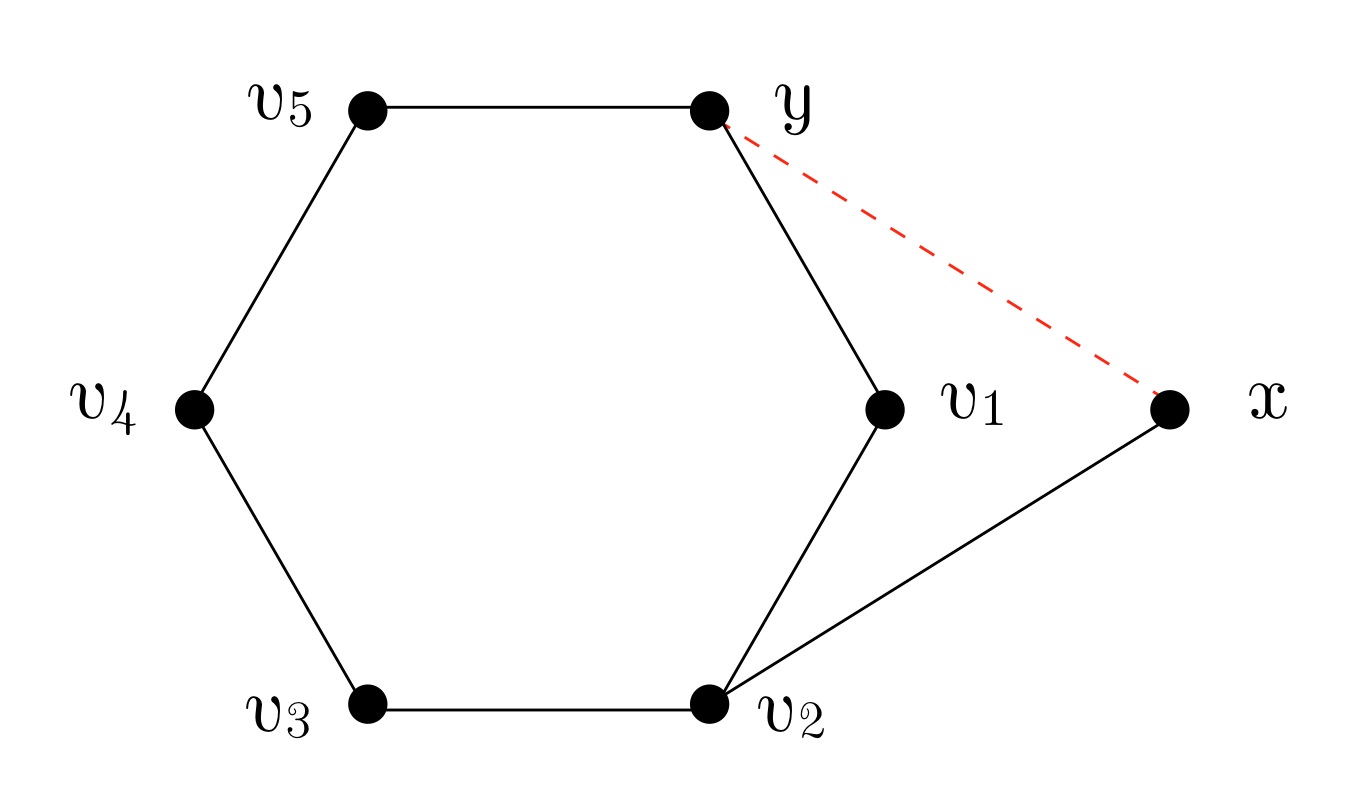}
\caption{$J_7 \ra (K_3+e,J_4)$}
\label{fig:hexProof}
\end{center}
\end{figure}

We will call a graph $G$ \emph{unsplittable} if $G \ra (K_3, J_4)$,
otherwise $G$ is \emph{splittable}. Our approach to obtain
$(K_3,K_3,J_4)$-colorings is based on Lemma \ref{lem:J7}
(here, a weaker arrowing $J_7 \ra (K_3,J_4)$ would suffice),
which implies that all such colorings can be produced from
a splittable $(J_7, K_3)$-good graph.

\begin{lemma} \label{lem:split}
If $m$ is the largest order of all  splittable $(J_7,K_3)$-good graphs, then
$R(K_3, K_3, J_4) = m + 1.$
\end{lemma}

\begin{proof}
By Lemma \ref{lem:J7}, the complement of the red subgraph
(union of green and blue subgraphs) of any $(K_3, K_3, J_4; n)$-coloring
is a splittable $(J_7, K_3; n)$-good graph. This shows
$R(K_3, K_3, J_4) \leq m + 1$. The edges of the complement
of a splittable graph $G$ of order $m$ give the red part of
a $(K_3, K_3,J_4; m)$-coloring, while any witness to the
splittability of $G$ defines the other two colors.
This shows $R(K_3, K_3, J_4) \geq m + 1$.
\end{proof}

Using the argument in the proof of Lemma \ref{lem:split} we can construct all $(K_3,K_3,J_4;n)-$colorings by splitting every $(J_7, K_3;n)$-good graph. The full set $\mathcal{R}(K_3, J_7)$ has been enumerated \cite{K3J7comp}\cite{MPR}, and $R(K_3, J_7)=21$ \cite{GH}. We independently computed $\mathcal{R}(K_3,J_7)$ using a simple vertex by vertex extension algorithm that generates $\mathcal{R}(K_3, J_7;n+1)$ from $\mathcal{R}(K_3, J_7;n)$, and utilizes the program {\em nauty} \cite{McK}\cite{McK2} to eliminate graph isomorphs. Our results agreed exactly with previously reported data shown in Table 1.\\

\begin{center}
  \begin{tabular}{@{} |r|r|r|@{}}
    \hline
    $n$ & $|\mathcal{R}(K_3,J_7;n)|$ & \#edges\\
    \hline
    1 & 1 & 0\\
    2 & 2 & 0-1\\
    3 & 3 & 0-2 \\
    4 & 7 & 0-4 \\
    5 & 14 & 0-6 \\
    6 & 38 & 0-9 \\
    7 & 105 & 2-12\\
    8 & 392 & 3-16 \\
    9 & 1697 & 4-20 \\
    10 & 9430 & 5-25\\
    11 & 58522 & 8-30 \\
    12 & 348038 & 11-36 \\
    13 & 1323836 & 15-36 \\
    14 & 2447170 & 19-40 \\
    15 & 1358974 & 24-45 \\
    16 & 158459 & 30-48 \\
    17 & 4853 &  37-50\\
    18 & 225 &  43-51\\
    19 & 1 &  54\\ 
    20 & 1 &  60\\
    \hline

\end{tabular}
\end{center}
\begin{center}
{Table 1:}
Statistics of $\mathcal{R}(K_3,J_7)$\\
\end{center}

\medskip
None of the complements of
graphs in $\mathcal{R}(K_3,J_7;n)$ for $n\geq17$ could be split into a $(K_3,K_3,J_4)$-coloring, which implies the following theorem.

\begin{theorem} \label{thm:33J}
$R(K_3,K_3,J_4)=17$.
\end{theorem}

\begin{proof}
We determined all splittable $(J_7,K_3)$-good graphs
of maximal order via two independent computational methods. First, for each $(J_7,K_3;n)$-good graph $G$ we created a conjunctive normal form (CNF) Boolean formula $\phi (G)$ which is satisfiable iff $G\not\rightarrow (K_3,J_4)$. The satisfiability of $\phi(G)$ was tested using a standard SAT-solver. In the second method we implemented our own computer algorithm which exhaustively searched through all relevant edge colorings.

Neither of the two methods found any splittable $(J_7, K_3; n)$-good graphs for $n \geq 17$,
and both found the same 11813 splittable $(J_7,K_3;16)$-good graphs.
So, by Lemma \ref{lem:split}, $R(K_3, K_3, J_4) = 17$.
Below we give further details about each method.\\

 \noindent{\bf Splittability via Satisfiability:}\\
If $G$ is a $(J_7, K_3)$-good graph, we wish to see if $G \ra (K_3, J_4)$. We consider each edge of $G$ to be a Boolean variable, and our colors as $F$ and $T$. We define the clauses of $\phi(G)$ as follows:
\begin{itemize}
\item For each $K_3$ with edges $\{e_1, e_2, e_3\}$ include the clause $(e_1\vee e_2\vee e_3)$. This forces at least one edge to have color $T$, so no $K_3$ will be formed in color $F$.
\item For each $J_4$ with edges $\{e_1, e_2, e_3, e_4, e_5\}$ include the clause $(\overline {e_1}\vee \overline{e_2} \vee \overline{ e_3} \vee \overline{e_4} \vee \overline{e_5})$. This forces at least one edge to have color $F$, so no $J_4$ will be formed in color $T$.
\end{itemize}
Clearly, the resulting $\phi(G)$ is satisfiable if and only if $G$ is splittable. We used the SAT-solver PicoSAT \cite{Bie}, the gold medal winner of the 2007 International SAT Competition in the industrial category, and found that no $(J_7,K_3;17)$-good graphs were splittable.\\

\noindent{\bf Recursive Coloring:}\\
We implemented the function $f(uncolored, green, blue)$
that takes three graphs as input. It attempts to take an
uncolored edge and add it to the current set of green edges
or blue edges, and recurse. If either recursion is successful,
{\bf True} is returned. In this way,
$f(E(G), \emptyset, \emptyset)$ returns {\bf True}
if $G$ can be split into a $(K_3, J_4)$-good graph
using the following algorithm.\\
\smallskip

\noindent$f(uncolored, green, blue) =$\\
\indent {\bf False} if $green$ contains a $K_3$
or $blue$ contains a $J_4$\\
\indent {\bf True} if $uncolored$ is empty\\
\indent {\bf Else} let $\{i,j\}$ be an edge in $uncolored$,\\
\indent \indent return $f(uncolored \setminus \{i,j\}, green \cup \{i,j\}, blue)$ \\
\indent \ \ \ \ \ \ \ \ \ \ \  $\vee f(uncolored \setminus \{i,j\}, green, blue \cup \{i,j\})$\\
\end{proof}

It could be tempting to obtain Theorem 4 by a simpler
approach of splitting $(K_7,K_3)$-good graphs.
However, the number of such graphs is much larger
than $(J_7,K_3)$-good graphs, and it seems infeasible
even just to enumerate the set $\mathcal{R}(K_7,K_3)$.

In another attempt to construct $\mathcal{R}(K_3,K_3,J_4;n)$ for $n\geq17$ we tried to enumerate $\mathcal{R}(K_6,J_4)$, since $K_6 \ra (K_3,K_3)$ and thus splitting $(K_6,J_4)$-good graphs leads to all $(K_3,K_3,J_4)$-colorings.
For more than 12 vertices the number of $(K_6,J_4)$-good graphs became too large to handle.
The attempt was continued by extending only suitably selected $(K_6,J_4;12)$-good graphs.
Eventually, all $6817238$ $(K_6,J_4;19)$- and $24976$ $(K_6,J_4;20)$-good
graphs were constructed, and none were found on 21 vertices,
confirming the previously unpublished results by McNamara that
$R(K_6,J_4)=21$ \cite{McN}. No $(K_6,J_4;19)$-good graphs could
be split into a $(K_3,K_3,J_4)$-coloring,
proving $R(K_3,K_3,J_4)\leq19$. However, the attempt
to enumerate $\mathcal{R}(K_6,J_4;18)$ was computationally infeasible. \\

\section {More Bounds}

\begin{theorem}
$21\leq R(K_3,J_4,J_4)\leq 27$.
\end{theorem}
\begin{proof}
The lower bound is established by a $(K_3,J_4,J_4;20)$-coloring presented in Figure \ref{fig:lb21}. It was obtained by splitting the unique $(J_7,K_3;20)$-good graph.

\begin{figure}[H]
\begin{center}
\begin{verbatim}
           0 2 2 3 3 2 2 3 2 3 3 3 2 2 1 1 1 1 1 1 
           2 0 3 2 3 1 2 3 3 2 1 2 3 1 2 1 1 1 3 2 
           2 3 0 3 2 1 3 2 2 1 2 3 1 3 1 2 1 1 3 2 
           3 2 3 0 2 1 3 2 1 2 3 1 3 2 1 1 2 1 2 3 
           3 3 2 2 0 1 1 1 3 3 2 2 2 3 1 1 1 2 2 3 
           2 1 1 1 1 0 3 2 3 2 2 2 3 3 2 2 3 3 1 1 
           2 2 3 3 1 3 0 2 3 3 2 1 1 1 3 2 2 1 2 1 
           3 3 2 2 1 2 2 0 1 1 1 3 3 2 2 3 3 1 1 2 
           2 3 2 1 3 3 3 1 0 2 3 2 1 1 2 3 1 2 2 1 
           3 2 1 2 3 2 3 1 2 0 2 1 2 1 3 1 3 2 3 1 
           3 1 2 3 2 2 2 1 3 2 0 1 1 2 1 3 2 3 3 1
           3 2 3 1 2 2 1 3 2 1 1 0 2 3 3 2 1 3 1 2 
           2 3 1 3 2 3 1 3 1 2 1 2 0 2 2 1 2 3 1 3 
           2 1 3 2 3 3 1 2 1 1 2 3 2 0 1 2 3 2 1 3 
           1 2 1 1 1 2 3 2 2 3 1 3 2 1 0 3 2 3 2 3 
           1 1 2 1 1 2 2 3 3 1 3 2 1 2 3 0 3 2 2 3 
           1 1 1 2 1 3 2 3 1 3 2 1 2 3 2 3 0 2 3 2 
           1 1 1 1 2 3 1 1 2 2 3 3 3 2 3 2 2 0 3 2 
           1 3 3 2 2 1 2 1 2 3 3 1 1 1 2 2 3 3 0 2 
           1 2 2 3 3 1 1 2 1 1 1 2 3 3 3 3 2 2 2 0 
\end{verbatim}
\caption {A $(K_3,J_4,J_4;20)$-coloring} 
\label{fig:lb21}
\end{center}
\end{figure}

For the upper bound, consider the graph $G$ formed by the green edges
of any $(K_3,J_4,J_4)$-coloring. By Lemma \ref{lem:J7}, $G$ must be a
$(J_4,J_7)$-good graph. $R(J_4,J_7)=28$, and it is known that there
exists a unique $(J_4,J_7;27)$-good graph \cite{McR}. This is the well
known strongly 10-regular Schl\"{a}fli graph \cite{Sei}. Reducing graph
splittability to Boolean satisfiability as in Section \ref{sec:33J4},
we determined that the complement of the Schl\"{a}fli graph is
unsplittable, and thus $R(K_3,J_4,J_4)\leq27$.
\end{proof}

We note that, interestingly, the same Schl\"{a}fli graph can be split
into two $J_4$-free graphs, which establishes the bound
$R_3(J_4) \ge 28$ \cite{Exoo}.

\begin{theorem}
$33\leq R(J_4,J_4,K_4)$.
\end{theorem}
\begin{proof}
The lower bound is established by a $(J_4,J_4,K_4;32)$-coloring
presented in Figure \ref{fig:lb33}. This coloring was found
using a standard simulated annealing algorithm.

\begin{center}
\begin{figure}[H]
{\small
\begin{verbatim}
    0 1 3 3 2 3 3 2 2 1 3 2 2 1 2 3 3 1 3 2 1 1 1 2 2 1 1 1 3 2 3 3
    1 0 3 1 2 1 3 2 1 3 1 2 1 3 3 1 3 3 3 1 2 3 2 1 1 2 2 3 3 2 1 2
    3 3 0 2 2 2 1 1 3 3 3 2 3 2 2 1 1 3 3 2 3 2 2 1 1 3 2 1 1 3 1 1
    3 1 2 0 3 3 2 3 2 3 3 2 2 3 3 3 1 1 2 3 3 1 1 2 3 1 1 2 1 2 2 1
    2 2 2 3 0 1 3 3 1 2 2 3 2 3 1 1 2 1 3 1 1 3 1 1 3 3 1 2 3 3 2 3
    3 1 2 3 1 0 1 1 3 3 2 3 2 3 2 2 1 2 3 3 3 3 1 2 2 1 3 1 1 2 3 1
    3 3 1 2 3 1 0 3 1 1 2 3 1 1 2 2 2 1 3 3 3 1 3 3 2 1 2 3 3 2 1 2
    2 2 1 3 3 1 3 0 3 1 3 1 1 2 3 2 3 2 2 1 1 3 3 2 1 2 1 3 2 3 3 3
    2 1 3 2 1 3 1 3 0 1 2 1 3 2 2 2 3 3 3 3 1 2 2 3 1 2 2 1 3 1 3 1
    1 3 3 3 2 3 1 1 1 0 2 2 3 2 1 1 1 2 1 1 2 3 1 1 3 2 2 3 2 3 3 3
    3 1 3 3 2 2 2 3 2 2 0 1 3 3 1 1 1 1 2 2 1 3 1 2 3 3 3 1 3 3 3 1
    2 2 2 2 3 3 3 1 1 2 1 0 3 3 3 2 3 2 1 3 3 1 1 1 2 1 1 3 1 3 3 3
    2 1 3 2 2 2 1 1 3 3 3 3 0 2 1 3 1 2 2 3 3 3 1 1 3 2 2 1 1 1 3 2
    1 3 2 3 3 3 1 2 2 2 3 3 2 0 1 3 1 3 1 1 2 3 3 1 2 3 1 3 1 2 2 2
    2 3 2 3 1 2 2 3 2 1 1 3 1 1 0 3 2 3 2 3 2 1 3 3 1 1 3 2 3 3 1 1
    3 1 1 3 1 2 2 2 2 1 1 2 3 3 3 0 3 3 1 2 2 1 3 3 3 1 1 3 1 1 2 2
    3 3 1 1 2 1 2 3 3 1 1 3 1 1 2 3 0 3 3 2 1 1 2 2 3 2 3 3 3 1 2 3
    1 3 3 1 1 2 1 2 3 2 1 2 2 3 3 3 3 0 1 2 3 2 2 1 1 3 3 2 1 1 2 3
    3 3 3 2 3 3 3 2 3 1 2 1 2 1 2 1 3 1 0 2 1 2 2 3 2 3 3 1 2 3 1 1
    2 1 2 3 1 3 3 1 3 1 2 3 3 1 3 2 2 2 2 0 2 3 3 3 3 2 2 1 1 2 1 1
    1 2 3 3 1 3 3 1 1 2 1 3 3 2 2 2 1 3 1 2 0 2 3 2 3 1 3 2 1 3 1 3
    1 3 2 1 3 3 1 3 2 3 3 1 3 3 1 1 1 2 2 3 2 0 3 2 1 2 3 1 3 2 2 2
    1 2 2 1 1 1 3 3 2 1 1 1 1 3 3 3 2 2 2 3 3 3 0 2 1 3 3 2 3 2 1 3
    2 1 1 2 1 2 3 2 3 1 2 1 1 1 3 3 2 1 3 3 2 2 2 0 3 1 2 3 3 3 3 1
    2 1 1 3 3 2 2 1 1 3 3 2 3 2 1 3 3 1 2 3 3 1 1 3 0 2 2 2 3 3 2 3
    1 2 3 1 3 1 1 2 2 2 3 1 2 3 1 1 2 3 3 2 1 2 3 1 2 0 3 2 3 1 3 3
    1 2 2 1 1 3 2 1 2 2 3 1 2 1 3 1 3 3 3 2 3 3 3 2 2 3 0 3 2 3 1 1
    1 3 1 2 2 1 3 3 1 3 1 3 1 3 2 3 3 2 1 1 2 1 2 3 2 2 3 0 2 3 3 2
    3 3 1 1 3 1 3 2 3 2 3 1 1 1 3 1 3 1 2 1 1 3 3 3 3 3 2 2 0 2 2 2
    2 2 3 2 3 2 2 3 1 3 3 3 1 2 3 1 1 1 3 2 3 2 2 3 3 1 3 3 2 0 1 1
    3 1 1 2 2 3 1 3 3 3 3 3 3 2 1 2 2 2 1 1 1 2 1 3 2 3 1 3 2 1 0 3
    3 2 1 1 3 1 2 3 1 3 1 3 2 2 1 2 3 3 1 1 3 2 3 1 3 3 1 2 2 1 3 0
\end{verbatim}
}
\caption {A $(J_4,J_4,K_4;32)$-coloring} 
\label{fig:lb33}
\end{figure}
\end{center}
\end{proof}

\section {Future Work}
Our work answers some of the open questions of Arste, Klamroth,
and Mengersen \cite{AKM}, while others remain open and should be studied
more. In particular, we think that
further progress on the known bounds for $R(K_3,J_4,J_4)$ and
$R(J_4,J_4,J_4)$ is feasible, definitely more so than for $R(K_3,K_3,K_4)$.
Another interesting project would be to study 3-color Ramsey numbers with
the parameters as in this paper, but in addition with at least one
color avoiding $C_4$.

\bigskip
\bigskip
\noindent
{\bf Acknowledgement.} We would like to thank the
anonymous reviewers whose suggestions led to improved
presentation of this work.

\end{document}